\documentclass[12pt]{extarticle}
\usepackage{amsmath, amsthm, amssymb, color}
\usepackage[colorlinks=true,linkcolor=blue,urlcolor=blue]{hyperref}
\usepackage{graphicx}
\usepackage{caption}
\usepackage{mathtools}
\usepackage{enumerate}
\usepackage{verbatim}
\usepackage{tikz,tikz-cd,tikz-3dplot}
\usepackage{amssymb}
\usetikzlibrary{matrix}
\usetikzlibrary{arrows}
\usetikzlibrary{positioning}
\usepackage{algorithm}
\usepackage[noend]{algpseudocode}
\usepackage{caption}
\usepackage[normalem]{ulem}
\usepackage{subcaption}
\oddsidemargin \evensidemargin
\usepackage{makecell}
\usepackage{array}

\newtheorem{theorem}{Theorem}
\newtheorem{proposition}[theorem]{Proposition}
\newtheorem{lemma}[theorem]{Lemma}
\newtheorem{corollary}[theorem]{Corollary}

\theoremstyle{definition}
\newtheorem{definition}[theorem]{Definition}

\newtheorem{conjecture}[theorem]{Conjecture}

\newtheorem{example}[theorem]{Example}

\numberwithin{theorem}{section}

\newcommand{\PP}{\mathbb{P}}
\newcommand{\RR}{\mathbb{R}}

\newcommand{\CC}{\mathbb{C}}

\newcommand{\minmax}{\mathop{{\rm min/max}}}

\title{\bf Grassmann and Flag Varieties in Linear Algebra, Optimization, and Statistics:\\ An Algebraic Perspective}
\author{Hannah Friedman and Serkan Ho\c{s}ten}
\date{}

\begin{document}
\maketitle

\begin{abstract}
  \noindent
 Grassmann and flag varieties lead many lives in pure and applied mathematics. Here we focus on the algebraic complexity of solving various problems in linear algebra and statistics as optimization problems over these varieties. The measure of the algebraic complexity is the amount of complex critical points of the corresponding optimization problem.   
 After an exposition of different realizations of these manifolds as algebraic varieties we present a sample of optimization problems over them and we compute their algebraic complexity. 
\end{abstract}

\section{Introduction}
A flag manifold or flag variety parametrizes nested sequences of subspaces  of fixed dimension, known as flags, of a finite dimensional vector space $V$. In this article, we take this vector space to be $V= \RR^n$ or $V=\CC^n$. We fix $r$ positive integers $1 \leq k_1 < k_2 < \cdots < k_r \leq n$ and consider 
$${\rm Fl}({\bf k};n) ={\rm Fl}(k_1, \ldots, k_r;n) = \{W_1 \subset W_2 \subset \cdots \subset W_r : W_i \subseteq V, \, \dim(W_i)= k_i, \, 1\leq i\leq r\}.$$
The set of flags ${\rm Fl}({\bf k};n)$ is both a (real or complex) manifold and a (real or complex) algebraic variety. 
As an algebraic variety, it lives in some affine space $\mathbb{A}^N$ or projective space $\PP^N$. 
In this case, we would like to know the defining equations, i.e., the polynomials whose common zero set is ${\rm Fl}({\bf k};n)$.
We also ask to know algebraic invariants of ${\rm Fl}({\bf k};n)$ such as the dimension and the degree where the latter depends on the particular embedding.

In Section~\ref{sec:flaglives}, we introduce and discuss several realizations of flag varieties, namely, the Stiefel coordinates, Pl\"{u}cker embedding, projection embedding, and  isospectral model.
The Grassmannian ${\rm Gr}(k,n)$ of $k$-dimensional subspaces of $V$ is a flag variety with $r=1$. It also appears with different realizations such as its Pl\"ucker and projection embeddings \cite{DFRS, fulton}. 
For each realization of the flag variety, we prove that the natural equations defining these varieties generate prime ideals. 
We discuss the relationships between these different embeddings. 
All of these realizations have appeared in the literature; see \cite{DFRS, fulton, LY24b, YWL}.
We collect them in one place to help anyone, especially those from the applied and computational algebraic geometry and algebraic statistics communities, to get a foot in the door. 

Our theme is optimization on Grassmann and flag varieties. As a starting point, consider the eigenvalue problem for symmetric matrices. Let $A$ be an $n \times n$ real symmetric matrix and consider the problem of minimizing or maximizing the Rayleigh quotient: 
\begin{align} \label{Rayleigh}
    \minmax_{z \neq 0}
    & \,\,\, \frac{z^T A z}{z^Tz}.
\end{align}
It is known that the optimal value for this eigenvalue problem is the minimum (respectively, maximum) eigenvalue of $A$. A formulation which avoids the fraction in the objective function is the quadratic optimization problem over  $S^{n-1}$, the sphere of dimension $n-1$:
\begin{align} \label{eigenvalue-over-sphere}
    \minmax_{z \in S^{n-1}}
    & \,\,\, z^T A z.
\end{align}
The optimal solution to the above {\it eigenvalue problem} (\ref{Rayleigh}) and (\ref{eigenvalue-over-sphere}) is a unit eigenvector corresponding to the minimum/maximum eigenvalue. 
The critical points of this constrained optimization problem are precisely the unit eigenvectors of $A$. 
The choice of sign for each eigenvector gives us $2n$ critical points.

The eigenvalue problem has the following generalization to compute multiple eigenvectors at once  
where we replace the vector $z$ in (\ref{eigenvalue-over-sphere}) with an $n \times k$ matrix $Z$ consisting of orthonormal columns. We solve a quadratic optimization problem over the Stiefel manifold called the \emph{multi-eigenvector problem}:
\begin{align} \label{multi-eigenvector}
    \minmax_{Z^TZ = {\rm Id}_k}
    & \,\,\, {\rm trace}(Z^T A Z).
\end{align}
The critical points of this problem are the sets of orthogonal $k$-frames consisting of unit eigenvectors of $A$, up to the action
of the orthogonal group ${\rm O}(k)$ over the complex numbers. 

\begin{theorem}\label{thm:eigen}
  Let $A$ be a generic real symmetric $n \times n$ matrix and let $Z$ be an $n \times k$ variable matrix. 
  The algebraic set of complex critical points of the multi-eigenvector problem (\ref{multi-eigenvector}) is
  $$ \bigsqcup_{\{i_1, \ldots, i_k\} \in \binom{[n]}{k}} \{[u_{i_1} \, u_{i_2} \, \cdots \, u_{i_k}] Q \, : \, Q \in {\rm O}(k)\} $$ 
  where $u_1, \ldots, u_n$ is an orthonormal eigenbasis of $A$. This algebraic set is a disjoint union of $\binom{n}{k}$ varieties isomorphic to ${\rm O}(k)$; it has dimension $\dim({\rm O}(k)) = \binom{n}{k}$
  and degree $\deg({\rm O}(k)) \binom{n}{k}$.
\end{theorem}
The degree of the special orthogonal group was computed in \cite[Theorem 1.1]{BBBKR}.
Because the multi-eigenvector problem  is ${\rm O}(k)$-invariant, it is more naturally considered over the Grassmannian. 
In Section~\ref{sec:eigenvalue}, we will restate this problem over the Grassmannian in projection coordinates \cite{DFRS}, where it has $\binom{n}{k}$ critical points (Theorem~\ref{thm:eigen-pgr-degree}).
To formulate this problem over the flag variety, one can use isospectral coordinates \cite{LY24b}. 
In the isospectral formulation, the set of critical points is a disjoint union of products of smaller orthogonal groups (Theorem~\ref{thm:eigen-iso-degree}).
In both Theorems~\ref{thm:eigen-pgr-degree} and \ref{thm:eigen-iso-degree}, we give explicit descriptions of the critical points.

When the set of critical points is finite, its size, the {\em algebraic degree}, provides a measure for the complexity of the problem. In particular, any optimal solution is an algebraic function of the input data of this degree; see \cite{JR09, JRS10,R12}.

In Section~\ref{sec:heterogeneous}, we study the heterogeneous quadratics minimization problem, which is used as a benchmark for testing numerical methods for Riemannian optimization \cite{oviedo2022}. 
Given symmetric $n \times n$ matrices $A_1, \ldots, A_k$, we seek to find $Z=[Z_1 \cdots Z_k]$ such that
\begin{align} \label{heteroquads}
    \minmax_{Z^TZ = {\rm Id}_k}
    & \,\,\, \sum_{i=1}^k Z_i^T A_i Z_i.
\end{align}
If $A_1 = \cdots = A_k$, one recovers the multi-eigenvector problem \eqref{multi-eigenvector}. 
However, for a generic choice of $A_1, \ldots, A_k$, this problem is not an optimization problem over the Grassmannian, but rather over the flag variety ${\rm Fl}(1, \ldots, k; n)$. 
We formulate the problem in projection coordinates and observe that in this formulation the heteregoneous  quadratics minimization problem is a generic linear optimization problem over the flag variety. 
We use numerical methods to compute the critical point counts for small values of $k$ (Table \ref{table:heteroquads}). 
We also report our observation based on our computational experiments that taking $A_1, \ldots, A_k$ to be generic diagonal matrices does not change the number of critical points. For $n=3$ and $k=2$, we show that this number is $40$ (Proposition \ref{prop:diagonal-3-2}).
The code for these computations is available at \url{https://github.com/hannahfriedman/flag_optimization_algebraic_perspective}.

In Section~\ref{sec:statisticsproblems}, we examine two problems from statistics \cite{YWL}: canonical correlation analysis and correspondence analysis. 
In these problems, we consider a rectangular matrix $A$ instead of a symmetric matrix and the critical points of the problems are given by the singular value decomposition of $A$. In the first case, we give a formula for the finite number of critical points which correspond to pairs of left and right singular vectors of $A$ (Theorem \ref{thm:canonical}). 
The second problem can be naturally formulated over a product of Grassmannians, and we give a complete description of the critical points in this case as well (Theorem \ref{thm:svd}).

\section{The Many Lives of the Flag Variety}\label{sec:flaglives}
We start with the observation 
that the set of partial flags $0 = W_0 \subset W_1 \subset W_2 \subset \cdots \subset W_r \subset V$ with $\dim(W_i) = k_i$ are in bijection with the 
sequence of subspaces 
$(W_1/W_0, W_2/W_1, W_3/W_2, \ldots, W_r/W_{r-1})$ where $W_i/W_{i-1}$ is viewed as a subspace in $V/W_{i-1}$ for $i=1, \ldots, r$. This gives a formula for the dimension of ${\rm Fl}({\bf k};n)={\rm Fl}(k_1, \ldots, k_r;n)$.
\begin{proposition} \label{dim-formula}
 The dimension of ${\rm Fl}({\bf k};n)$ equals $\sum_{i=1}^r (k_i - k_{i-1})(n-k_i)$. In particular, the dimension of the complete flag variety ${\rm Fl}(1, 2, \ldots,n-1; n)$ is $n(n-1)/2$.   
\end{proposition}
\begin{proof} The above bijection 
implies that there is a
bijection from ${\rm Fl}({\bf k};n)$ to the points
in 
$${\rm Gr}(n,k_1) \times {\rm Gr}(k_2-k_1, n-k_1) \times \cdots \times {\rm Gr}(k_r-k_{r-1}, n-k_{r-1}).$$
Since the $i$th Grassmannian above has dimension $(k_i-k_{i-1})(n-k_{i-1})$ the formula follows. 
For the complete flag variety $k_i - k_{i-1} = 1$ for all $i =1, \ldots, r$, and again the formula follows. 
\end{proof}

The flag variety ${\rm Fl}({\bf k};n)$ is first and foremost conceived as a real or complex manifold. We can associate to a flag $W_1 \subset W_2 \subset \cdots \subset W_r \subset V$ an orthogonal matrix $Q$ whose first $k_i$ columns comprise an orthonormal basis of $W_i$ for $i = 1, \ldots, r$. The smaller orthogonal group ${\rm O}(k_i-k_{i-1})$ acts on the columns $Q_{k_{i-1}+1}, \ldots, Q_{k_i}$ by right multiplication without altering the flag. Therefore
$${\rm Fl}(k_1, \ldots, k_r; n) \, \simeq \, 
{\rm O}(n) \, / \, {\rm O}(k_1) \times {\rm O}(k_2-k_1) \times \cdots \times {\rm O}(n-k_r).$$
It is instructive to work out that the dimension of the quotient on the right hand side 
$$\binom{n}{2} - \binom{k_1}{2} - \binom{k_2-k_1}{2} - \cdots - \binom{k_r-k_{r-1}}{2} - \binom{n-k_r}{2}$$
gives the formula for the dimension of ${\rm Fl}({\bf k};n)$ in Propositon \ref{dim-formula}. A related formulation utilizes the {\it Stiefel manifold} $\mathcal V_{k,n}$. 
\begin{definition} The Stiefel manifold $\mathcal V_{k,n}$ is the set of orthonormal $k$-frames. Namely, it consists of $n \times k$ matrices $Z$ where $Z^TZ = {\rm Id}_k$.
\end{definition}
The Stiefel manifold $\mathcal V_{k,n}$ is an affine algebraic variety in $\mathbb{A}^{nk}$.
When $k=n$, $\mathcal V_{k,n} = {\rm O}(n)$ and it has two irreducible components. 
Some of the contents of the following theorem are known, for example in \cite[Lemma 2.4]{BG}.
\begin{theorem} \label{thm:stiefel}
The Stiefel manifold $\mathcal V_{k,n}$ has dimension
$$\dim(\mathcal V_{k,n}) = nk - \binom{k+1}{2} \, = \, \binom{n}{2} - \binom{n-k}{2}. $$  
The Stiefel ideal
$I_{\rm St} = I(\mathcal V_{k,n}) = \langle Z^TZ - {\rm Id}_k \rangle$ is a complete intersection.
When $k < n$, $\mathcal V_{k,n}$ is irreducible and $I_{\rm St}$ is prime. 
\end{theorem}
\begin{proof}
 For the dimension, we observe that ${\rm Fl}(1,2,\ldots,k;n) \simeq \mathcal V_{k,n} \, / \, \prod_{i=1}^k {\rm O}(1)$, since each orthonormal frame $Z$ corresponds to a flag 
$W_1 \subset W_2 \subset \cdots \subset W_k$ with $\dim(W_i) = i$ up to the sign of each column of $Z$. This gives a finite map from $\mathcal V_{k,n}$ to ${\rm Fl}(1,2,\ldots, k;n)$ of degree $2^k$. Hence, the dimension of $\mathcal V_{k,n}$
is equal to the dimension of this flag manifold which is $\binom{n}{2} - \binom{n-k}{2} = nk - \binom{k+1}{2}$ by
Proposition \ref{dim-formula}.
The affine variety defined by $I_{\rm St}$ is precisely $\mathcal V_{k,n}$.
Since $I_{\rm St}$ has $\binom{k+1}{2}$ generators, it is a complete intersection. 

The special orthogonal group ${\rm SO}(n)$ acts on $\mathcal V_{k,n}$ transitively for $k<n$, and this shows that $\mathcal V_{k,n}$ is irreducible in this case. 
The variety $\mathcal V_{k,n}$ is the orbit of $[e_1 \, \cdots \, e_k]$ under this action where $e_j$ is the $j$th standard unit vector in $\CC^n$. We show that the affine scheme defined by $I_{\rm St}$ is reduced by showing that the above point is smooth. 
The Jacobian of $I_{\rm St}$ is 
\begin{align}\label{eq:stiefel-jac}
    \mathrm{Jac}(Z)^T \, = \, \left( \begin{array}{cccccccccccccc}
     2Z_1 &  &    &  &   & Z_2  & Z_3 & \cdots & Z_k &     &       &  & \cdots&  \\
          & 2Z_2 & & &   & Z_1  &     &        &     & Z_3 &\cdots &Z_k&  &   \\
          &     & 2Z_3&  &    & &Z_1 &        &     & Z_2 &      &     & \cdots&   \\ 
          &   &  & \ddots&     & & &   \ddots&     &     &\ddots &     &  &Z_k \\
      &    &     &  & 2Z_k&    & &           & Z_1 &     &       & Z_2 & \cdots & Z_{k-1}
    \end{array} \right),
  \end{align}
and $\mathrm{Jac}([e_1 \, \cdots \, e_k])^T$  has rank $\binom{k+1}{2}$ since all columns have different supports. This shows that $I_{\rm St}$ is a radical ideal in general and it is a prime ideal for $k<n$. 
\end{proof}
Now, we can associate to a flag $W_1 \subset W_2 \subset \cdots \subset W_r \subset V$ with $\dim(W_i) = k_i$ an orthonormal $k_r$-frame $Z$ together with the orthogonal groups ${\rm O}(k_i-k_{i-1})$ acting on the columns $Z_{k_{i-1} +1}, \ldots, Z_{k_i}$ by right multiplication. In other words, 
$${\rm Fl}(k_1, \ldots, k_r; n) \, \simeq \, 
\mathcal V_{k_r,n} \, / \, {\rm O}(k_1) \times {\rm O}(k_2-k_1) \times \cdots \times {\rm O}(k_r-k_{r-1}).$$
In particular, the Grassmannian ${\rm Gr}(k,n)$
is identified with $\mathcal V_{k,n} \, / \, {\rm O}(k)$. 
For both the flag and Grassmann manifolds, we  call their representation as a quotient of the Stiefel variety the representation in {\it Stiefel coordinates}. We will use the Stiefel coordinates for the Grassmannian to formulate the multi-eigenvector problem. 

The above initial introduction views  ${\rm Fl}({\bf k};n)$ and ${\rm Gr}(k,n)$ primarily as manifolds. As algebraic varieties, they have different embeddings. The classical embedding is in projective space. The case for the Grassmanian ${\rm Gr}(k,n)$ is well-known. A $k$-dimensional subspace $W$ of $V$ is identified with a $k\times n$ matrix $A$ of rank $k$ whose rows form a basis of $W$. This identification is up to the action of ${\rm GL}(k)$ by left multiplication to account for all possible bases of $W$. Then we consider the map
$${\rm Gr}(k,n) \longrightarrow \mathbb{P}^{\binom{n}{k} - 1} \quad \quad \quad A \mapsto \det(A_I)_{I \in \binom{[n]}{k}}$$
where $A_I$ is the $k \times k$ submatrix of $A$
whose columns are indexed by $I$. The image is an irreducible projective variety of dimension $k(n-k)$ defined by the quadratic {\it Pl\"ucker relations} \cite[Section 9.1]{fulton} in the coordinates $x_I$ of $\mathbb{P}^{\binom{n}{k}-1}$. In similar fashion, we embed ${\rm Fl}(k_1, \ldots, k_r;n)$ into $\mathbb{P}^{\binom{n}{k_1}-1} \times \cdots \times \PP^{\binom{n}{k_r}-1}$. This time we consider the map

\begin{align} \label{flag-pluecker}
{\rm Fl}({\bf k}; n) \to \mathbb{P}^{\binom{n}{k_1}-1} \times \cdots \times \PP^{\binom{n}{k_r}-1}, \quad
 A \mapsto \left(\det(A_{I_1})_{I_1 \in \binom{[n]}{k_1}},\, \cdots ,\, \det(A_{I_r})_{I_r \in \binom{[n]}{k_r}}\right)
\end{align}
where $A$ is a $k_r \times n$ matrix of rank $k_r$ and $A_{I_j}$ is the $k_j \times k_j$ submatrix of the first $k_j$ rows of $A$ with columns indexed by $I_j$. The image is an irreducible projective variety of dimension $\sum_{i=1}^r (k_i-k_{i-1})(n-k_i)$ also defined by quadratic relations in the coordinates $x_I$ where $I$ is a subset of $n$ of cardinality equal to one of $k_1, k_2, \ldots, k_r$. The equations below express the fact that
$W_s \subset W_t$ for every pair of subspaces in the flag whenever $s < t$.
\begin{theorem}[Proposition 9.1.1, \cite{fulton}] \label{thm:pluecker}
    The variety ${\rm Fl}(k_1, \ldots, k_r; n) \subseteq  \mathbb{P}^{\binom{n}{k_1}-1} \times \cdots \times \PP^{\binom{n}{k_r}-1}$ in Pl\"{u}cker coordinates is defined by the prime ideal generated by the quadrics
    \begin{align*}
        x_{i_1, \ldots, i_{k_s}} 
        x_{j_1, \ldots, j_{k_t}}
        - \sum x_{i'_1, \ldots, i'_{k_s}} x_{j'_1, \ldots, j'_{k_t}}
    \end{align*}
    for every pair of $1\leq s < t \leq r$ and 
    where the sum is over all $(i', j')$ obtained by exchanging the first $m$ of the $j$-subscripts with $m$ of the $i$-subscripts while preserving their order and given a permutation $\sigma \in S_{k_s}$, $x_{i_1, \ldots, i_{k_s}} = {\rm sgn}(\sigma) x_{\sigma(i_1), \ldots, \sigma(i_{k_s})}$.
\end{theorem}

We now turn to realizations of Grassmann and flag varieties in applied settings where they are represented by symmetric matrices. 
We give two such representations of flag varieties. 

The first one uses projection matrices. 
One can uniquely represent a $k$-dimensional subspace of $\RR^n$ with the orthogonal projection matrix $P$  onto that subspace. 
The symmetric matrix $P$ satisfies $P^2 =P$ and ${\rm trace}(P) = k$. 
These equations realize the Grassmannian as an affine variety in ${\mathbb A}^{\binom{n+1}{2}}$ which we call the 
 \emph{projection Grassmannian} ${\rm pGr}(k,n)$ \cite{DFRS,LY24a}. 

\begin{proposition}[Theorem 5.2, \cite{DFRS}] The projection Grassmannian ${\rm pGr}(k,n)$ is an irreducible variety defined by the prime ideal
$$ \langle P^2 - P \rangle + \langle {\rm trace}(P) -  k \rangle \,\, \subset  \,\, \CC[P].$$  
\end{proposition}
We extend this definition to flag varieties. 
Indeed, if $P$ and $P'$ are projection matrices with ${\rm rank}(P) \leq {\rm rank}(P')$, then ${\rm im}(P) \subseteq {\rm im} (P')$ precisely when $P'P = P$.
We therefore define the \emph{projection flag variety} as 
\begin{align}\label{eq:projectionflag1}
    {\rm pFl}({\bf k}; n) = \{(P_1, \ldots, P_r) \in {\rm Sym}(\CC^n)^r \colon P_iP_j = P_j,\, {\rm trace}(P_i) = k_i \textrm{ for $1 \leq j \leq i\leq r$}\}. 
\end{align}
This definition may be found in \cite[Proposition 17]{YWL}.
The affine variety ${\rm pFl}({\bf k};n)$ is irreducible: given an orthonormal frame $Z$ in the Stiefel variety $\mathcal V_{k_r,n}$, we set $P_i = [Z_1 
 \cdots Z_{k_i}] [Z_1 \cdots Z_{k_i}]^T$. This yields projection matrices satisfying $P_iP_j = P_j$ for $j\leq i$
 and ${\rm trace}(P_i)=k_i$. Conversely, given such projection matrices, the column span of $P_i$ is the subspace $W_i$ in the flag. Hence, we can construct an orthonormal frame $Z \in \mathcal V_{k_r,n}$ such that $Z_1, \ldots, Z_{k_i}$ is an orthonormal basis of $W_i$, and now $P_i = [Z_1 
 \cdots Z_{k_i}] [Z_1 \cdots Z_{k_i}]^T$. Since ${\rm pFl}({\bf k};n)$ is given by this parametrization we obtain the result:
 \begin{proposition}
  The projection flag variety ${\rm pFl}({\bf k}; n)$ is an irreducible affine variety.  
 \end{proposition}
The equations in (\ref{eq:projectionflag1}) define the projection flag variety set theoretically.
We now prove that they generate the prime ideal $I({\rm pFl}({\bf k}; n))$. We note that in \eqref{eq:projectionflag1} it suffices to retain the relations $P_i^2 = P_i$ for $i = 1, \ldots, r$ and  $P_iP_{i-1}=P_{i-1}$ for $i=2,\ldots, r$. 

\begin{theorem}\label{thm:projectionideal}
    Let ${\rm Fl}(k_1, \ldots, k_r; n)$ be a flag variety and let $P_1, \ldots, P_r$ be $n \times n$ symmetric matrices of unknowns.
    Then ${\rm pFl}({\bf k}; n)$ is smooth and has prime ideal
    \begin{align*}
          \langle P_{i}P_{i-1} - P_{i-1} : 2 \leq i \leq r \rangle + 
        \langle P_i^2 - P_i, {\rm trace}(P_i) - k_i : 1 \leq i \leq r\rangle \subset \CC[P_1, \ldots, P_r].
    \end{align*}
\end{theorem}
\begin{proof}
    We follow the proof of the Grassmannian case in \cite[Theorem 5.2]{DFRS}.
    Since the variety is irreducible and the generators of the ideal cut out the variety set theoretically, it suffices to prove that the scheme of the ideal is smooth. 
    As in the Grassmannian case, there is a transitive ${\rm O}(n)$ action on tuples $(P_1, \ldots, P_r)$ and hence we need only check smoothness at the point $E_{\bf k} = (E_{k_1}, \ldots, E_{k_r})$ where $E_i$ is the $n \times n$ matrix with $(E_i)_{\ell j} = 1$ if $\ell = j \leq i$ and zero otherwise.
    The codimension of ${\rm pFl}({\bf k}; n)$ is $N = \sum_{i=1}^r \binom{n+1}{2} - n(n-k_i) + \sum_{i=2}^r k_{i-1}(n - k_i)$.
    We prove that the rank of the Jacobian is at least $N$ by describing $N$ linearly independent  rows of the Jacobian evaluated at $E_{\bf k}$. 

    Fix $i \in \{1, \ldots, r\}$ 
    and choose $\ell, j \in \{1, \ldots, k_i\}$ or $\ell, j \in  \{k_i+1,\ldots, n\}$. 
    Then
    \begin{align*}
        \left. \frac{\partial (P_{i}^2 - P_{i})_{\ell j}}{\partial (P_{i'}
)_{\ell'j'}}\right|_{E_{\bf k}} = 
        \begin{cases}
            \pm 1 &\quad \textrm{$i' = i, \ell' = \ell, j' = j$},\\
            0 & \quad \textrm{else}\\
        \end{cases}
    \end{align*}
    where the sign is positive if $\ell, j \leq k_i$ and negative if $k_i < \ell, j$.
    There are $\binom{n+1}{2} - k_i(n - k_i)$ choices for such pairs $\{\ell, j\}$.

    Next fix $i \in \{2, \ldots, r\}$, let $j \in \{1, \ldots, k_{i-1}\}$, and let $\ell \in \{k_{i}+1, \ldots, n\}$.
    Then 
    \begin{align*}
        \left. \frac{\partial (P_{i}P_{i-1} - P_{i-1})_{\ell j}}{\partial (P_{i'}
)_{\ell'j'}}\right|_{E_{\bf k}} = 
        \begin{cases}
            1 &\quad \text{if $i' = i, \ell' = \ell, j' = j$},\\
            -1 &\quad \text{if $i' = i-1, \ell' = \ell, j' = j$},\\
            0 &\quad \text{else.}
        \end{cases}
    \end{align*}
    There are $k_{i-1}(n - k_i)$ choices for pairs $\{\ell, j\}$.
    Ranging over all choices of $i$ yields $N$ rows of the Jacobian with disjoint supports. 
    Thus the Jacobian has the desired rank.
\end{proof}
The second representation of Grassmann and flag varieties using symmetric matrices is known as the \emph{isospectral model} \cite{LY24b}. 
This representation encodes a flag with a single matrix.
To construct this matrix, we represent a flag in ${\rm Fl}({\bf k}; n)$ by an equivalence class of orthogonal matrices $Q \in {\rm O}(n)$ modulo an ${\rm O}(k_1) \times \cdots \times {\rm O}(k_r - k_{r-1}) \times {\rm O}(n-k_r)$-action. 
We then define a diagonal matrix $X_0 = {\rm diag}(c_1, \ldots, c_n)$ such that $c_1 = \cdots = c_{k_1}$, $c_{k_1+1} = \cdots = c_{k_2}$, etc. 
The symmetric matrix $X = Q^TX_0Q$ does not depend on our choice of representative $Q$ and therefore uniquely encodes the flag. 
So the flag variety ${\rm Fl}({\bf k}; n)$ is embedded as the variety of diagonalizable complex symmetric matrices with spectrum ${\bf c} = (c_1, \ldots, c_n)$. 
\begin{definition}
    Suppose $c_{{k_i}+1}= \cdots = c_{k_{i+1}}$ for $i=0,\ldots,r$ 
    where $k_0 =0$ and $k_{r+1} =n$ and let $X_0 = {\rm diag}({\bf c})$.
    The \emph{isospectral model} for the flag variety ${\rm Fl}_{\bf c}({\bf k}; n)$ is the image of the map
    \begin{align}\label{eq:iso-param}
        {\rm SO}(n) \to \CC^{n \times n}, \quad\quad Q \mapsto QX_0Q^T.
    \end{align}
\end{definition}
From this definition, we see that the isospectral model is irreducible of the correct dimension. 
A complex symmetric $X$ is diagonalizable precisely when  its minimal polynomial has no repeated factors. 
By \cite[Theorem XI.4]{gantmacher}, diagonalizable implies orthogonally diagonalizable,
so the condition that the minimal polynomial of $X$ has no repeated factors suffices to prove that $X$ is in ${\rm Fl}_{\bf c}({\bf k}; n)$ for some choice of ${\bf c}$. 
Supposing $X$ satisfies the polynomial $\prod_{j=1}^r (X - c_{k_j}{\rm Id}_n)$, we know the eigenvalues of $X$ are  $c_{k_j}$, but we do not know the multiplicities. 
For this task, we need ${\rm trace}(X) = \sum_{j=1}^nc_j$. 
To identify the multiplicities for non-generic ${\bf c}$, one must consider the traces of powers of $X$, as well. 
However, it is shown in \cite{LY24b} that for generic ${\bf c}$, ${\rm trace}(X)$ gives enough information to determine the multiplicities of the eigenvalues.
Thus a matrix $X$ is in ${\rm Fl}_{\bf c}({\bf k}; n)$ precisely when $\prod_{j=1}^r (X - c_{k_j}{\rm Id}_n) = 0$ and ${\rm trace}(X) = \sum_{j=1}^nc_j$. 
This discussion shows that these equations cut out the isospectral model over $\CC$. 
We now prove that they generate a prime ideal.

\begin{theorem}\label{thm:iso_ideal}
    Let ${\rm Fl}({\bf k};n)$ be a flag variety and let $X$ be a symmetric matrix of unknowns. 
    Given a generic choice of $c_1, \ldots, c_n$ satisfying $c_{k_j+1} = \cdots =c_{k_{j+1}}$ for $j=0, \ldots, r$, the variety ${\rm Fl}_{\bf c}({\bf k}; n)$ is smooth and its prime ideal is
    \begin{align*}
        I({\rm Fl}_{\bf c}({\bf k};n)) = 
        \langle \prod_{j = 1}^r (X - c_{k_j}{\rm Id}_n), \,\,
        {\rm trace}(X) - \sum_{j = 1}^n c_j \rangle.
    \end{align*}
\end{theorem}
\begin{proof}
    We proceed as in the proofs of Theorem \ref{thm:stiefel} and Theorem \ref{thm:projectionideal} and show that the affine scheme defined by the right hand side ideal is smooth at every closed point. This proves smoothness of the variety and primeness of the ideal. 
    
    The orthogonal group ${\rm O}(n)$ acts transitively on closed points by conjugation. 
    Thus every closed point is conjugate to the matrix $X_0 = {\rm diag}(c_1, \ldots, c_n)$.
    Hence, we check smoothness at $X_0$ using the Jacobian criterion. 
    It suffices to check that the Jacobian $J({\bf c})$ evaluted at $X_0$
    has rank at least $n + \binom{k_1}{2} + \cdots + \binom{n-k_r}{2}$. 
    The rows and columns of $J({\bf c})$ are both indexed by entries $(i,j)$ of the matrix $X$. 
    We will show that the entry $J({\bf c})_{(i,j), (i', j')} \neq 0$ precisely when $i = i', j = j'$, and $c_i = c_j$.
    With the correct ordering of variables, $J({\bf c})$ is diagonal with diagonal entries of the form $J({\bf c})_{(i,j), (i,j)}$.
    It is then clear from counting that this matrix has the correct rank. 

    We prove off-diagonal entries of $J({\bf c})$ are $0$. 
    If $\{i,j\} \cap \{i',j'\} = \emptyset$, by the product rule 
    \begin{align*}
       \left. \frac{\partial(X^m)_{ij}}{\partial x_{i'j'}}  \right|_{X = X_0} = 
       \left. \sum_{\ell=1}^n x_{i\ell} \frac{\partial (X^{m-1})_{\ell j}}{\partial x_{i'j'}} \right |_{X = X_0}
       =
       \left. c_i \frac{\partial (X^{m-1})_{i j}}{\partial x_{i'j'}} \right |_{X = X_0} = 0
    \end{align*}
    where the second equality comes from evaluating at $X = X_0$ and the last equality is induction on $m$.
    Since derivatives of all powers of $X$ vanish,  $J({\bf c})_{(i,j), (i',j')} = 0$ if $\{i,j\} \cap \{i',j'\} = \emptyset$.

    Next suppose $i = i'$ and $j \neq j'$. Again using the product rule, 
    \begin{align*}
       \left. \frac{\partial(X^m)_{ij}}{\partial x_{ij'}}  \right|_{X = X_0} = 
       \left. (X^{m-1})_{j'j} +  \sum_{\ell=1}^n x_{i\ell} \frac{\partial (X^{m-1})_{\ell j}}{\partial x_{ij'}} \right |_{X = X_0}=
       \left.   c_{i} \frac{\partial (X^{m-1})_{i j}}{\partial x_{ij'}} \right |_{X = X_0} = 0
    \end{align*}
    where the last equality is by induction on $m$.
    As above, we have $J({\bf c})_{(i,j),(i,j')}=0$ for $j \neq j'$. 
    This concludes the proof that $J({\bf c})$ is a diagonal matrix.

    Finally, we prove that the diagonal entries $J({\bf c})_{(i,j), (i,j)}$ are nonzero if  $c_i = c_j$. We compute
    \begin{align*}
        \left. \frac{\partial(X^m)_{ij}}{\partial x_{ij}}  \right|_{X = X_0} = 
       \left. (X^{m-1})_{jj} +  \sum_{\ell=1}^n x_{i\ell} \frac{\partial (X^{m-1})_{\ell j}}{\partial x_{ij}} \right |_{X = X_0}=
       \left .c_j^{m-1} + c_i \frac{\partial (X^{m-1})_{i j}}{\partial x_{ij}}\right |_{X = X_0}.
    \end{align*}    
    By induction we see that
  $        \left. \frac{\partial(X^m)_{ij}}{\partial x_{ij}}  \right|_{X = X_0} = \sum_{p=0}^{m-1} c_i^p c_j^{m-1- p}.$
  If $c_i = c_j$, $\frac{\partial(X^m)_{ij}}{\partial x_{ij}} = mc_i^m$ and so $J({\bf c})_{(i,j),(i,j)}  = f'(c_i)$ where $f(z) = (z - c_{k_1}) \cdots (z - c_{k_r})$ is the minimal polynomial of $X$.  
     Since $c_i$ is a root of $f$ and $f$ has no double roots $J({\bf c})_{(i,j),(i,j)} \neq 0$.
    
    Since there are at least $n + \binom{k_1}{2} + \cdots +\binom{n-k_r}{2}$ nonzero entries on the diagonal of $J({\bf c})$ and all off-diagonal entries are zero, we conclude that $J({\bf c})$ has rank at least $n + \binom{k_1}{2} + \cdots +\binom{n-k_r}{2}$. 
\end{proof}
\begin{example} 
    We show the different ways of representing the flag variety ${\rm Fl}(1, 2; 3)$. 
    In Stiefel coordinates, this flag variety is realized as $\mathcal V_{2,3} / {\rm O}(1) \times {\rm O}(1)$ via a $3 \times 2$ matrix 
    $$Z = \begin{pmatrix} z_{11} & z_{12} \\ z_{21} & z_{22} \\ z_{31} & z_{32} \end{pmatrix}$$ with orthonormal columns. We identify such matrices up to  sign of each column.

    To obtain the Pl\"{u}cker coordinates for the flag represented by $Z$, we take minors. 
    This yields six Pl\"ucker coordinates that realize this complete flag variety in $\mathbb{P}^2 \times \mathbb{P}^2$:
    \begin{align*}
    \begin{matrix}
        x_1 = z_{11}, & x_2 = z_{21}, & x_3 = z_{31},\\
        x_{12} = z_{11}z_{22} - z_{12}z_{21},
        & 
        x_{13} = z_{11}z_{32} - z_{12}z_{31},
        &
        x_{23} = z_{21}z_{32} - z_{22}z_{31}.
        \end{matrix}
    \end{align*}
    By Theorem \ref{thm:pluecker} these coordinates satisfy a single relation $x_1 x_{23} - x_2 x_{13} + x_3 x_{12} = 0$.

    To represent the flag in projection coordinates, let $P_1$ be the unique orthogonal projection matrix onto the space spanned by the first column of $Z$ and $P_2$ be the unique orthogonal matrix onto the column space of $Z$:
    \begin{align*}
        P_1 &= 
        \begin{pmatrix}
            z_{11} \\ z_{21} \\ z_{31}
        \end{pmatrix}
        \begin{pmatrix}
            z_{11} &z_{21} & z_{31}
        \end{pmatrix}
        , \quad \quad 
        P_2 = ZZ^T.  
    \end{align*}
    One checks that these matrices satisfy the relations
    \begin{align*}
        P_1^2 = P_1 &&  P_2^2 = P_2  && P_2P_1 = P_1\\
        {\rm trace}(P_1) = 1 &&  {\rm trace}(P_2) = 2.
    \end{align*}
    Finally, we can also represent this flag with a single matrix via isospectral coordinates. 
    We append a column to $Z$ to form a $3 \times 3$ orthogonal matrix $\tilde{Z}$ and fix the generic vector ${\bf c} = (c_1,c_2,c_3)$.
    The isospectral coordinates for $Z$ with respect to $\bf c$ are given by the matrix
    \begin{align*}
        S = \tilde{Z} \, {\rm diag}(c_1, c_2, c_3)\, \tilde{Z}^T
    \end{align*}
    whose eigenvalues are  $c_1, c_2, c_3$ and whose eigenvectors are the columns of $\tilde{Z}$.
    This matrix satisfies its minimal polynomial and has the correct trace: 
    \begin{align*}
        (S - c_1{\rm Id}_3)(S - c_2{\rm Id}_3)(S - c_3{\rm Id}_3) = 0 && 
        {\rm trace}(S) = c_1 + c_2 + c_3.
    \end{align*}
\end{example}
From this example, we see that it is not difficult to write the Pl\"{u}cker, projection, and isospectral coordinates once we have Stiefel coordinates for a flag. 
The following proposition summarizes what is known about how to write different coordinates for the flag variety in terms of other coordinates.

\begin{proposition}
  The relationships between the different lives of the flag  variety ${\rm Fl}(k_1, k_2, \ldots, k_r;n)$ are captured in Figure~\ref{fig:lives}. In the figure, $Z$ is an $n \times k_r$ matrix in $\mathcal V_{k_r,n}$, $\tilde{Z}$ is an $n \times n$ orthogonal matrix whose first $k_r$ columns comprise $Z$, and $Z_{I,J}$ denotes the submatrix of $Z$ with rows indexed by the set  $I$ and columns indexed by the set $J$. The variables $x_{i_1, \dots, i_{k_s}}$ indicate Pl\"{u}cker coordinates and $X_i$ denotes the cocircuit matrix for ${\rm Gr}(k_i, n)$ as in \cite[Corollary 2.5]{DFRS}. The matrix $P_i$ is a projection matrix in the tuple of 
  $(P_1, \ldots, P_r)$. Finally, $S$
  is an $n \times n$ symmetric matrix with eigenvalues ${\bf c} = (c_1, \ldots, c_n)$.
\end{proposition}

\begin{proof}
The map from Stiefel to Pl\"ucker coordinates  is as in (\ref{flag-pluecker}).
The map from Pl\"{u}cker to projection coordinates is explained for the Grassmannian in \cite[Corollary 2.5]{DFRS} and extends naturally to the flag variety. 

To move from Stiefel coordinates to projection coordinates, one takes projection matrices onto the spans of the appropriate submatrices of $Z$. 
To go the other way, one takes a spectral decomposition of $P_r = \tilde{Z} E_{k_r} \tilde{Z}^T$ where $E_{i}$ is  the $n \times n$ matrix whose $(1,1), \ldots, (i,i)$ entries are 1 and all other entries are zero. Set $Z$ to be the first $k_r$ columns of $ {\tilde Z}$.

The map from isospectral to Stiefel coordinates is similar: take the spectral decomposition of $S = \tilde{Z}\,{\rm diag}(c_1, \ldots, c_n)\, \tilde{Z}^T$ and remove the last $n - k_r$ columns from $\tilde{Z}$.
Given Stiefel coordinates $Z$, one forms $\tilde{Z}$ by extending the columns of $Z$ to an orthonormal basis for $\CC^n$. 
The equivalence between these coordinates follows from the definitions of the Stiefel and isospectral coordinates. 

To move from projection coordinates to isospectral coordinates one composes the maps passing through the Stiefel coordinates. 
Write the diagonal matrix 
\begin{align*}
    {\rm diag}(c_1, \ldots, c_n) = c_n {\rm Id}_n + (c_{k_r} - c_n) E_{k_r} + \cdots + (c_{k_1} - c_{k_2})E_{k_1}.
\end{align*}
Conjugating with the augmented Stiefel coordinates $\tilde{Z}$ yields the expression for $S$ in terms of projection coordinates.
\end{proof}

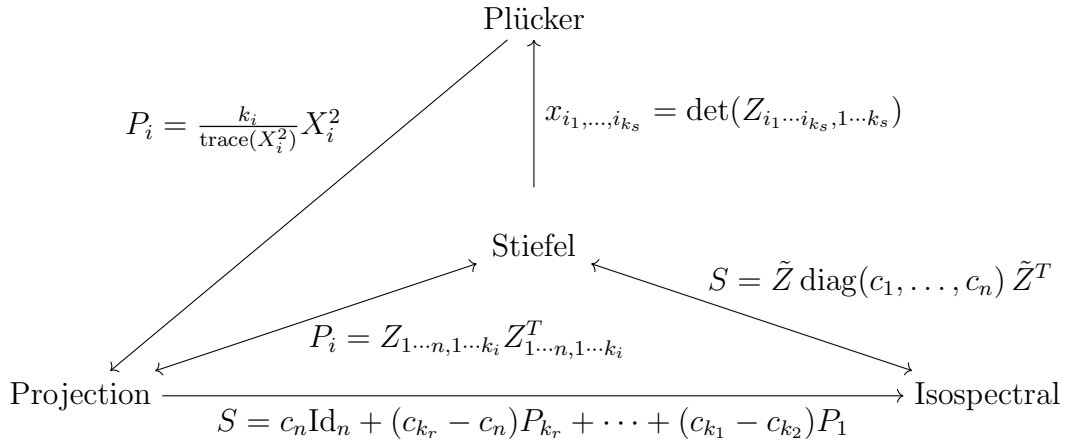
\begin{figure}[h]
  \begin{center}
\begin{tikzpicture}[node distance=2.5cm]
    \node[minimum size=1.5cm] at (0,0) (center) {\rm Stiefel};
    \node at (0, 3) (top) {\textrm{Pl\"{u}cker}};
    \node at (-6, -2) (left) {\rm Projection};
    \node at (6, -2) (right) {\rm Isospectral};

    \draw[->] (center) -- (top) node[midway, right] {$x_{i_1, \ldots, i_{k_s}} = \det(Z_{i_1 \cdots i_{k_s}, 1\cdots k_s})$};
    \draw[<->] (center) -- (left) node[midway,  xshift = 2cm, yshift=-0.3cm] {$P_i = Z_{1\cdots n, 1 \cdots k_i}Z_{1\cdots n, 1 \cdots k_i}^T$};;
    \draw[<->] (center) -- (right) node[midway, xshift = 1.7cm, yshift=0.5cm] {$S = \tilde{Z}\, {\rm diag}(c_1, \ldots, c_n) \,\tilde{Z}^T$};
    \draw[->] (left) -- (right) node[midway, below] {$S = c_{n} {\rm Id}_{n} + (c_{k_r} - c_{n}) P_{k_r} + \cdots + (c_{k_1} - c_{k_2}) P_1$};
    \draw[->] (top) -- (left) node[midway, xshift=-1cm, yshift=1cm] {$P_i = \frac{k_i}{{\rm trace}(X_i^2)} X_i^2$} ;   
\end{tikzpicture}
\end{center}  
  \caption{Diagram explaining how to move from one life of the flag variety to another. If $A \to B$ in the diagram, the edge label explains how to write the $B$ coordinates in terms of the $A$ coordinates. Two of the arrows are bidirectional, meaning that one direction comes from matrix multiplication and the other comes from a matrix factorization.}\label{fig:lives}
\end{figure}

\section{The Multi-Eigenvector  Problem}\label{sec:eigenvalue}
In this section, we give three different formulations of the 
multi-eigenvector problem (\ref{multi-eigenvector}). The first two are formulated as optimization problems over a Grassmannian in Stiefel and projection coordinates. The third one utilizes a flag variety in its isospectral formulation.  In all three cases we describe the sets of critical points of these optimization problems, where in the Stiefel case we describe the ${\rm O}(k)$-orbit of the critical points in $\mathcal V_{k,n}$. 

\subsection{Critical points in Stiefel coordinates}
The goal of this section is to prove Theorem \ref{thm:eigen} which characterizes the critical points of the multi-eigenvector problem in Stiefel coordinates. We first prove the following lemma.
\begin{lemma} \label{lem:diagonalization}
    Suppose $A\in \CC^{n \times n}$ is orthogonally diagonalizable and has full rank. 
    Let $Z \in \mathcal V_{k,n}$.
    If $M \in \CC^{k \times k}$ is such that $AZ = ZM$, then $M$ is also orthogonally diagonalizable. 
\end{lemma}
\begin{proof}
    Let $A = U\Lambda U^T$ be a factorization of $A$ with $\Lambda$ diagonal and $U \in {\rm O}(n)$. 
    Then $AZ = U\Lambda U^T Z = ZM$ and rearranging we have $M = (Z^T U) \Lambda (U^T Z)$ where $(Z^TU)(U^TZ) = {\rm Id}_k$. 
    This factorization implies that for every diagonal entry $\lambda_i$ of $\Lambda$, $M (Z^T U)_i = \lambda_i(Z^T U)_i$. 
    Since $M$ has at most $k$ eigenvalues and the $\lambda_i$ are nonzero, $n - k$ of the columns of $Z^T U$ are forced to be zero. 
    This yields an orthogonal diagonalization for $M$.
\end{proof}

\begin{proof}[Proof of Theorem \ref{thm:eigen}]
  Write $Z_i$ for the $i$th column of $Z$. 
  Note first that $\textrm{tr} (Z^T A Z) = \sum_{j = 1}^k Z_i^T A Z_i$.
  This allows us to compute
  \begin{align*}
    \nabla \textrm{tr} (Z^T A Z) = 2\begin{pmatrix} Z_1^TA && Z_2^TA && \cdots && Z_k^T A \end{pmatrix}
  \end{align*}
  which is a row vector with $nk$ entries.
  The Jacobian of the polynomials defined by $Z^TZ = \textrm{Id}_k$ is $ \mathrm{Jac}(Z)$ as in \eqref{eq:stiefel-jac}.
  A critical point is a vector
  ${Z}$ satisfying ${Z}^T{Z} = \textrm{Id}_k$ where $\nabla \textrm{tr}({Z}^TA{Z})$ is in the row span of $\mathrm{Jac}({Z})$. 
  In particular, there exist $\mu_{ij} \in \CC$, $1 \leq i, j \leq k$ with $\mu_{ij} = \mu_{ji}$ such that $A Z_i =  \sum_{j=1}^k \mu_{ij} Z_j$. 
  These $\mu_{ij}$ can be assembled into a $k \times k$ symmetric matrix $M$ satisfying $A{ Z} = ZM$. 
  So $ Z$ is a critical point if and only if it satisfies ${ Z}^T { Z} = {\rm Id}_{k}$ and there exists a symmetric matrix $M \in {\rm Sym}({\CC}^k)$ such that $A{Z} = { Z} M$.

  If such a matrix $M$ exists, then by Lemma~\ref{lem:diagonalization}, it is orthogonally diagonalizable. 
  Therefore we replace the condition $A{ Z} = { Z}M$ with the condition $A{ Z} =  Z Q^T \Lambda Q$. where $Q \in {\rm O}(k)$ and $\Lambda$ is diagonal.
  Thus $ Z$ is a critical point if and only if it satisfies $ Z^T Z= {\rm Id}_k$ and $A ZQ^T =  ZQ^T \Lambda$. 
  These conditions hold precisely when the columns of $ZQ^T$ are a subset of an orthonormal eigenbasis for $A$, which proves that $ Z$ has the desired form. 
\end{proof}

\subsection{Critical points in projection coordinates}
In Stiefel coordinates, we 
work with an orbit of critical points in $\mathcal V_{k,n}$
which corresponds to a $k$-dimensional subspace spanned by $k$ eigenvectors of $A$. 
It is desirable to formulate this problem so that it has finitely many critical points, because equivalence classes are difficult to compute with. 
We therefore transform the multi-eigenvector problem into projection coordinates for the Grassmannian as in Section \ref{sec:flaglives}. 
If $Z$ is $n \times k$ and $Z^TZ = {\rm Id}_k$, then $P = ZZ^T$ projects onto the column space of $Z$. 
Then using the trace trick, we have ${\rm trace}(Z^T A Z) = {\rm trace}(AP)$ and we can reformulate our problem as
\begin{align}\label{eq:eigen-pgr-problem}
    \minmax_{\substack{P^2 = P \\ {\rm trace}(P) = k}} \,\, {\rm trace}(AP).
\end{align}
We first prove a lemma relating critical points under a polynomial map. We will use this result a few times in the paper. 
\begin{lemma} \label{lem:critical} Let $V \subset \CC^s$ and $W \subset \CC^t$ be smooth affine varieties. Let $\phi \, : \, V \longrightarrow W$ be a surjective morphism, and let $f \in \CC[W]$ be
a regular function. Then $z \in V$ and $p=\phi(z)$
are critical points of $f' = f \circ \phi$ and $f$, respectively, if the differential $d\phi_z \, : \, T_z(V) \longrightarrow T_p(W)$ is surjective. 
\end{lemma}
\begin{proof} The differential $d\phi_z$ is, 
in coordinates, equal to the Jacobian of $\phi$ restricted to $T_z(V)$. By the chain rule, the gradient of $f'$ equals $\nabla_z (f') = \nabla_p(f) \cdot  {\rm Jac}_z \phi$, where $\nabla_p(f)$ is a row vector. Suppose $z \in V$ is a critical point of $f'$. This means that $\nabla_z(f') \in T_z(V)^{\perp}$,  
i.e., $\nabla_z(f') \cdot v = 0$ for all $v \in T_z(V)$.  Equivalently, $\nabla_p(f) \cdot {\rm Jac}_z\phi \cdot v =0$ for all $v \in T_z(V)$.  This implies that $\nabla_p(f)$ is orthogonal to the subspace of the tangent space $T_p(W)$ that is the image of $T_z(V)$ under the differential $d\phi_z$.   Now, if $d\phi_z$ is surjective, then it is guaranteed that  $\nabla_p(f)$ is in  $T_p(W)^{\perp}$, and hence $p$ is a critical point of $f$. For the converse, pick a critical point $p \in W$. So $\nabla_p(f)$ is in $T_p(W)^{\perp}$. 
Since $\phi$ is surjective, let $z$ be any point in the fiber $\phi^{-1}(p)$. By the surjectivity of the differential $d\phi_z$ the image of $T_z(V)$ under the differential is $T_p(W)$, and since  $\nabla_z(f') \cdot v =  \nabla_p(f) \cdot {\rm Jac}_z \phi \cdot v = 0$
for all $v \in T_p(V)$ we conclude that $z$ is a critical point of $f'$ for every $z \in \phi^{-1}(p)$. 
\end{proof}

\begin{theorem}\label{thm:eigen-pgr-degree}
    The degree of the multi-eigenvector problem \eqref{eq:eigen-pgr-problem} over the projection Grassmannian $\mathrm{pGr}(k, n)$ is $\binom{n}{k}$. 
    The critical points are the $k$-dimensional subspaces spanned by $k$ eigenvectors of $A$. 
\end{theorem}
\begin{proof}
   Consider the parametrization $\Phi$ from the Stiefel variety $\mathcal V_{k,n}$ to the projection Grassmannian which sends $Z \mapsto ZZ^T$. This is a surjective morphism between the two smooth varieties 
   (see the proof of Theorem 5.2 in \cite{DFRS} for the smoothness of the projection Grassmannian) where an entire orbit in Theorem \ref{thm:eigen} is sent to a single projection matrix.  Now we prove that the differential $d\Phi_Z$ is surjective.  Since ${\rm O}(n)$ acts transitively on $\mathcal V_{k,n}$, we consider the tangent space $T_{[e_1 \cdots e_k]}(\mathcal V_{k,n})$.
   We note that the columns of ${\rm Jac}(Z)$ as in (\ref{eq:stiefel-jac})
   are indexed by $z_{ij}$ 
   for $i=1,\ldots, n$ and $j=1, \ldots, k$. We see that the standard unit vectors $e_{ij}$ for $i=k+1, \ldots,n$ and $j=1, \ldots, k$ are in the kernel of ${\rm Jac}([e_1 \cdots e_k])$. We next compute the Jacobian of $\Phi$. For this, we order the $kn$ columns of ${\rm Jac}_Z\Phi$ by the entries of the rows of $Z$. 
   It will be more convenient to 
   look at blocks of entries in each row corresponding the rows of $Z$ which  we denote by $\tilde{Z}_1, \ldots, \tilde{Z}_n$. 
   The rows of ${\rm Jac}_Z\Phi$ correspond to the $\binom{n+1}{2}$ 
   entries of the symmetric matrix $ZZ^T$. We order the rows by first the diagonal entries of this matrix, then the first superdiagonal, followed by the second superdiagonal, etc. Then the first $n$ rows of ${\rm Jac}_Z \Phi$ is ${\rm diag}(2\tilde{Z}_1, \ldots, 2\tilde{Z}_n)$. The next $n-1$ rows will be of the form
   $$(0 \cdots 0 \, \tilde{Z}_{i+1} \, \tilde{Z}_i \, 0 \cdots 0)$$ for $i=1, \ldots, n-1$ where $\tilde{Z}_i$ appears at block $i+1$. The following $n-2$ rows are of the form
   $$(0 \cdots 0 \, \tilde{Z}_{i+2} \, 0  \, \tilde{Z}_i \, 0 \cdots 0)$$  
   for $i=1, \ldots, n-2$ where $\tilde{Z}_i$ appears at block $i+2$, etc. The last row will be 
   $(\tilde{Z}_n \, 0 \cdots 0 \, \tilde{Z}_1)$. Evaluating this Jacobian at $[e_1 \cdots e_k]$ produces $k(n-k)$ distinct standard unit vectors of $\CC^{kn}$ in the last $k(n-k)$ columns indexed by the entries 
   of $\tilde{Z}_{k+1}, \ldots, \tilde{Z}_n$. Hence, multiplying ${\rm Jac}_Z\Phi$ with the vectors $e_{ij}$ for $i=k+1, \ldots, n$ and $j=1, \ldots, k$ that are tangent vectors to $\mathcal V_{k,n}$ at $Z$
   gives $k(n-k)$ linearly independent vectors in the tangent space to ${\rm pGr}(k,n)$ at $ZZ^T$. Since the dimension of this smooth Grassmannian is $k(n-k)$ we conclude that $d\Phi_Z$ is surjective. We finish the proof by applying Lemma \ref{lem:critical}.
\end{proof}
The trace of the objective function is a generic linear function in the entries of $A$:
\begin{align*}
    {\rm trace}(AP) = \sum_i a_{ii} p_{ii} + 2\sum_{i < j} a_{ij}p_{ij}.
\end{align*}
Thus solving the multi-eigenvector problem on the projection Grassmannian is equivalent to solving a generic linear optimization problem on the same variety. 
The number of critical points of optimizing a generic linear form on a variety is called the \emph{linear optimization degree (LO degree)} of the variety \cite{MRWW}.
\begin{corollary}\label{cor:grassLO}
    The LO degree of $\mathrm{pGr}(k,n)$ is $\binom{n}{k}$.
\end{corollary}
\subsection{Critical points in isospectral coordinates}
We now present the most general formulation of the multi-eigenvector problem. 
While the Stiefel and projection formulations are optimization problems over the Grassmannian, writing this problem in isospectral coordinates allows us to generalize to an optimization problem over a flag variety ${\rm Fl}({\bf k}; n) = {\rm Fl}(k_1, \ldots, k_r; n)$. 
We fix $c_1, \ldots, c_n \in \CC$ with $c_1 = \cdots = c_{k_1}$, $c_{k_1+1} = \cdots = c_{k_2}$, $\cdots$, $c_{k_r+1} = \cdots = c_n$. We let $X_0 = {\rm diag}(c_1, \ldots, c_n)$. With this, we introduce the following  linear optimization problem over the isospectral model ${\rm Fl}_{\bf c}({\bf k};n)$
\begin{align}\label{eq:iso-eval}
    \minmax_{S \in {\rm Fl}_{\bf c}({\bf k}; n)} \,\, {\rm trace}(AS).
\end{align}
If the $c_i$ are distinct, then the optimization problem is over the complete flag variety and as we will see it has $n!$ critical points. 
On the other hand, if $c_1 = \cdots = c_k = 1$ and $c_{k+1} = \cdots = c_n = 0$, then $S = Q X_0 Q^T$ is a point in ${\rm pGr}(k,n)$ and we recover the optimization problem \eqref{eq:eigen-pgr-problem}.
The following theorem interpolates between these extreme cases. 

\begin{theorem}\label{thm:eigen-iso-degree}
  The linear optimization problem \eqref{eq:iso-eval} over the isospectral flag variety ${\rm Fl}_{\bf c}({\bf k};n)$ has $\binom{n}{k_1, k_2-k_1, \ldots, n-k_r}$ critical points.
\end{theorem}

We first prove the following which is an analog of Theorem \ref{thm:eigen}. 
\begin{theorem} \label{thm:eigen-iso-impl} 
    Let $A$ be a generic real symmetric $n \times n$ matrix and $X_0 = {\rm diag}(c_1, \ldots, c_n)$ with $c_{k_j+1} = \cdots = c_{k_{j+1}}$ for $j=0, \ldots, r$ where $k_0=0$ and $k_{r+1} = n$.
    The algebraic set of complex critical points of the optimization problem
    \begin{align} \label{eq:eigen-iso-param}
        \minmax_{Q^TQ = {\rm Id}_n }\,\, {\rm trace}(A QX_0Q^T)
    \end{align}
    is equal to 
    \begin{align}\label{eq:iso-cps}
     \bigsqcup_{\sigma \in  S_n/(S_{k_1} \times S_{k_2-k_1} \cdots \times S_{n-k_r})}
     \{ U(\sigma \cdot R) \, : \, R \in  {\rm O}(k_1) \times {\rm O}(k_2-k_1) \times \cdots \times {\rm O}(n - k_{r})\} 
     \end{align}
    where $U = [u_1 \, u_2 \, \cdots \, u_n]$ for $u_1, \ldots, u_n$ an orthonormal eigenbasis of $A$ and $\sigma$ acts by permuting the rows of $R$. This algebraic set is a disjoint union of 
    $\binom{n}{k_1, k_2-k_1, \ldots, n-k_r}$ 
    varieties isomorphic to 
    $ {\rm O}(k_1) \times {\rm O}(k_2-k_1) \times \cdots \times {\rm O}(n - k_{r}) $. 
\end{theorem}
\begin{proof}
  Similar to the proof of Theorem~\ref{thm:eigen}, a matrix $Q \in {\rm O}(n)$ is a critical point of the problem if and only if there exists a symmetric matrix $M$ such that $AQX_0 = QM$.
  The condition that $Q^TAQX_0$ is symmetric is equivalent to requiring that $Q^TAQ$ and $X_0$ commute, which occurs precisely when $Q^TAQ$ is a block diagonal matrix with blocks of size $k_1,(k_2-k_1), \ldots, (n-k_r)$.
 If $Q$ is in the disjoint union \eqref{eq:iso-cps}, then 
    \begin{align*}
       Q^T AQ = (R^T \cdot \sigma^{-1})U^T AU (\sigma \cdot R) = R^T \Lambda R
    \end{align*}
    where $\Lambda$ is the diagonal matrix with eigenvalues of $A$ on the diagonal. 
    Since $\Lambda$ is diagonal and $R$ has the desired block structure, the points in the disjoint union are critical points.

    Conversely, suppose $Q^TAQ$ has the desired block structure. 
    Since $A$ is diagonalizable, every block of $Q^T A Q$ is diagonalizable.
    Assembling these blocks gives a diagonalization $Q^T AQ = R \Lambda R^T$ where $R \in  {\rm O}(k_1) \times {\rm O}(k_2-k_1) \times   \cdots \times {\rm O}(n-k_r)$, up to a permutation of the rows of $R$.
    But then $A = QR \Lambda R^TQ^T$, which implies that $Q = UR^T$, as desired.    
\end{proof}
\begin{proof}[Proof of Theorem~\ref{thm:eigen-iso-degree}]
    The parameterization ${\rm SO}(n) \to {\rm Fl}_{\bf c}({\bf k}; n)$ can naturally be extended to $\Phi \colon {\rm O}(n) \to {\rm Fl}_{\bf c}({\bf k}; n)$. 
    The fiber of a point in ${\rm Fl}_{\bf c}({\bf k}; n)$ is isomorphic to a copy of ${\rm O}(k_1) \times \cdots \times {\rm O}(n - k_r)$. 
    Once we prove that the parametrization takes critical points to critical points, we will be finished as in the proof of Theorem \ref{thm:eigen-pgr-degree}.
    Indeed, we may apply Lemma~\ref{lem:critical}, since the parametrization is surjective and $ {\rm Fl}_{\bf c}({\bf k}; n)$ is smooth by Theorem~\ref{thm:iso_ideal}.     

    Since there is a transitive ${\rm O}(n)$-action on ${\rm O}(n)$, it suffices to compute the dimension of the image of the Jacobian evaluated at the tangent space of the identity matrix ${\rm Id}_n$. 
    For all $1 \leq i < j \leq n$ we have $e_{ij} - e_{ji} \in T_{{\rm Id}_n}({{\rm O}(n)})$. 
    These $\binom{n}{2}$ vectors are linearly independent. 

    We now compute ${\rm Jac}_{{\rm Id}_n}(\Phi)$. 
    This matrix has rows of the form $2c_i e_{ii}^T$ for all $i = 1, \ldots, n$ and of the form $c_je_{ij}^T + c_i e_{ji}^T$ for $1 \leq i < j \leq n$. 
    Therefore ${\rm Jac}_{{\rm Id}_n}(\Phi) \cdot (e_{ij} - e_{ji})$ is zero when $c_i = c_j$ and otherwise has one nonzero entry. 
    Since the nonzero entry has a different index for distinct pairs $i,j$, this process produces $\dim({\rm Fl}_{\bf c}({\bf k}; n))$ linearly independent vectors. 
\end{proof}
\begin{corollary}
    The LO degree of the isospectral model ${\rm Fl}_{\bf c}({\bf k};n)$ is $\binom{n}{k_1, k_2-k_1, \ldots, n-k_r}$.
\end{corollary}
We close this section with an observation that came to our attention after the first version of our paper was posted, and that gives an alternate proof of Theorem~\ref{thm:eigen-iso-degree}. For this, we recall the {\it Euclidean distance degree} of a variety $X \subset \mathbb{R}^n$ \cite{DHOST}. It is the algebraic degree, namely, the number of complex critical points of 
\begin{align} 
 \min_{x \in X} & \, \, \, || u - x ||^2
\end{align}
where $u \in \mathbb{R}^n$ is a generic point. 
\begin{proposition} [Proposition 3.1, \cite{LLY}] The Euclidean distance degree of the isospectral model ${\rm Fl}_{\bf c}({\bf k};n)$ is $\binom{n}{k_1, k_2-k_1, \ldots, n-k_r}$. 
\end{proposition}
\begin{corollary} The degree of the multi-eigenvector problem 
over the isospectral flag variety ${\rm Fl}_{\bf c}({\bf k};n)$ is equal to its Euclidean distance degree.
\end{corollary}
\begin{proof} The Euclidean distance minimization problem over 
${\rm Fl}_{\bf c}({\bf k};n)$ is
\begin{align} 
 \min_{S \in {\rm Fl}_{\bf c}({\bf k}; n)} & \, \, \, || A - S ||^2
\end{align}
where $A$ is a generic symmetric matrix. We note that 
$|| A - S||^2 = {\rm trace}(A^2) - 2  {\rm trace}(AS) + {\rm trace}(S^2)$. Both ${\rm trace}(A^2)$ and ${\rm trace}(S^2) = {\rm trace}(X_0^2)$ are constant, and the result follows. 
\end{proof}

\section{Heterogeneous Quadratics Minimization Problem}\label{sec:heterogeneous}
The heterogeneous quadratics minimization problem generalizes the multi-eigenvector problem considered in the previous section. Its most natural formulation is in Stiefel coordinates:
\begin{align} \label{heterogeneous}
 \min_{Z^TZ = {\rm Id}_k} &\sum_{i=1}^k Z_i^T A_i Z_i
\end{align}
where $A_i \in {\rm Sym}(\RR^n)$ and $Z_i$ denotes the $i$th column of $Z$ for $i = 1, \ldots, k$. The multi-eigenvector problem is recovered by choosing $A_i=A$ of all $i=1, \ldots, k$. 

In Stiefel coordinates, the critical points
can be computed using
\begin{align*}
    \nabla \left( \sum_{i=1}^k Z_i^T A_i Z_i \right)= 2\begin{pmatrix} Z_1^TA_1 & Z_2^TA_2 & \cdots & Z_k^TA_k \end{pmatrix}
  \end{align*}
and $\mathrm{Jac}(Z)$ as in \eqref{eq:stiefel-jac}. A more convenient system of equations whose solutions are the same critical points is the following (see \cite[Lemma 1]{oviedo2022}):
$$ [A_1Z_1 \, A_2Z_2 \, \cdots \, A_kZ_k]Z^T - Z[A_1Z_1 \, A_2Z_2 \, \cdots \, A_kZ_k]^T = 0 \mbox{   and    } 
Z^TZ = {\rm Id}_k.$$ 
The optimization problem is invariant under the action of ${\rm O}(1)^k$: if $[Z_1 \, Z_2 \, \cdots \, Z_k]$ is a critical point, so is $[\pm Z_1 \, \pm Z_2 \, \cdots \, \pm Z_k]$. Computing the number of complex critical points of (\ref{heterogeneous}) is a challenging open problem. 
Our numerical computations, produced with \verb+HomotopyContinuation.jl+ \cite{hc}, are summarized in Table~\ref{table:heteroquads}. 
\begin{table}[h] 
    \[
    \begin{array}{|c|c|c|c|c|c|c|c|c|}
    \hline
        & n = 2 &  n = 3 & n = 4 & n=5 & n = 6 & n = 7 & n = 8 & n = 9\\
              \hline
        k = 2 & 8 & 40 & 112 & 240 & 440 & 728 & 1120 & 1632\\
        \hline
        k = 3 & &  80 & 960 & 5536 & 21440 & 64624 & &\\
        \hline
        k = 4 &  &  & 1920 & 57216 & & & & \\
        \hline
    \end{array}\]
    \caption{Degrees of the heterogeneous quadratics minimization problem for small $k, n$.}\label{table:heteroquads}
\end{table}
\begin{conjecture}
The number critical points of the heterogeneous quadratics minimization problem for $k=2$ is $8  \sum_{j=1}^{n-1} j^2.$
\end{conjecture}

\subsection{Diagonal case}
Our computational experiments indicate that the number of critical points of the heterogeneous quadratics minimization problem stays stable if we take the input matrices $A_1, \ldots, A_k$ to be generic {\it diagonal} matrices. While we do not have a general proof for this observation, we present a result addressing the first nontrivial case. 
\begin{proposition} \label{prop:diagonal-3-2}
 Let $A_1 = \mathrm{diag}(a_{11}, a_{12}, a_{13})$ and $A_2 = \mathrm{diag}(a_{21}, a_{22}, a_{23})$ be generic diagonal matrices. Then the algebraic degree of the corresponding heterogeneous quadratics minimization problem \eqref{heterogeneous} is $40$.   
\end{proposition}
\begin{proof}
We will explicitly describe these $40$ critical points. The critical points are defined by the Lagrange multiplier equations 
$$A_1 Z_1 = q_{11} Z_1 + q_{12}  Z_2, \,\, A_2 Z_2 = q_{12} Z_1 + q_{22} Z_2, \quad Z_1^T Z_1 = 1,\,
Z_1^T Z_2 = 0,\, Z_2^T Z_2 = 1$$
where $q_{11}, q_{12}, q_{22}$ are Lagrange multipliers. 
This is a square system with $9$ variables and $9$ equations.
We obtain $2^2 \cdot 3 \cdot 2 = 24$ solutions by taking $Z_1^* = \pm e_i$, $Z_2^* = \pm e_j$, $q_{11}^* = a_{1i}$, $q_{12}^* = 0$, and $q_{22}^* = a_{2j}$ for all $i \neq j$. 
By computing a Gr\"obner basis of the ideal given by the above equations over the rational function field $\mathbb{Q}(a_{1j}, a_{2j}: \, j=1,2,3)$ we see that this ideal is zero dimensional and has $40$ solutions. In addition to the $24$ solutions we already described, the rest of the $16$ solutions come via row and columns sign flips of
the ${Z^*}$ whose rows we list below:
\begin{align*}
    ({Z^*}^T)_1 = \frac{\sqrt{a_{12} - a_{13} + a_{23} - a_{22}}}{\alpha}
    \begin{pmatrix}
        \sqrt{-(a_{12} - a_{13})(a_{21} - a_{22})(a_{21} - a_{23})}\\
        \sqrt{(a_{22} - a_{23})(a_{11} - a_{12})(a_{11} - a_{13})}
    \end{pmatrix}\\
    ({Z^*}^T)_2 = \frac{\sqrt{a_{11} - a_{13} + a_{23} - a_{21}}}{\alpha}
    \begin{pmatrix}
        \sqrt{-(a_{11} - a_{13})(a_{22} - a_{21})(a_{22} - a_{23})}\\
        \sqrt{(a_{21} - a_{23})(a_{12} - a_{11})(a_{12} - a_{13})}
    \end{pmatrix}\\
    ({Z^*}^T)_3 = \frac{\sqrt{a_{11} - a_{12} + a_{22} - a_{21}}}{\alpha}
    \begin{pmatrix}
        \sqrt{-(a_{11} - a_{12})(a_{23} - a_{21})(a_{23} - a_{22})}\\
        \sqrt{(a_{21} - a_{22})(a_{13} - a_{11})(a_{13} - a_{12})}
    \end{pmatrix}
\end{align*}
where $\alpha =a_{11}a_{22} - a_{11}a_{23} + a_{12}a_{23} - a_{12}a_{21} + a_{13}a_{21} - a_{13}a_{22}$.
The corresponding Lagrange multipliers are 
\begin{align*}
    q^*_{11} 
    &= \frac{1}{\alpha}\big (-a_{11}a_{12}(a_{21} - a_{22}) + a_{11}a_{13}(a_{21} - a_{23}) - a_{12}a_{13}(a_{22} - a_{23})\big )\\
    q^*_{12} 
    &=  \frac{2}{\alpha}\sqrt{-(a_{11} - a_{12})(a_{11} - a_{13})(a_{12} - a_{13})(a_{21} - a_{22})(a_{21} - a_{23})(a_{22} - a_{23})}\\
    q^*_{22}&= \frac{1}{\alpha} \big (a_{21}a_{22}(a_{11} - a_{12}) - a_{21}a_{23}(a_{11} - a_{13}) + a_{22}a_{23}(a_{12} - a_{13})\big ).
\end{align*}
These computations were performed in \verb+Oscar.jl+ \cite{oscar}.
\end{proof}
\subsection{Projection coordinates}
We now formulate the heterogeneous quadratics minimization problem as an optimization problem over a flag variety in projection coordinates.
We rewrite the objective function as
\begin{align*}
   \sum_{i = 1}^n Z_i^T  A_i Z_i = \sum_{i = 1}^n {\rm trace}(Z_i^T  A_i Z_i) = \sum_{i = 1}^k \textrm{trace}(B_i P_i) 
\end{align*}
where $P_i = \sum_{j=1}^i Z_j Z_j^T$ for $i =1 ,\ldots, k$ and $B_i = A_i - A_{i+1}$ for $i = 1, \ldots, k-1$ and $B_k = A_k$. 
Over $\mathrm{pFl}(1,2,\ldots, k;n)$ we can reformulate our problem as follows:
\begin{align} \label{heterogeneous-projection}
  \textrm{minimize } &\sum_{i=1}^k \textrm{trace}(B_iP_i)\\
  \nonumber
  \textrm{subject to } 
  &P_iP_j = P_j,\,\,\textrm{trace}(P_i) = i \textrm{ for } 1\leq j \leq  i \leq k.
\end{align}

\begin{proposition}
    If the heterogeneous quadratics minimization problem \eqref{heterogeneous-projection} has $m$ critical points, then \eqref{heterogeneous} has $2^k m$ critical points in $\mathcal V_{k,n}$.
\end{proposition}
\begin{proof}
The map from the Stiefel formulation of the flag variety $\mathrm{Fl}(1,2,\ldots, k;n)$ to the projection formulation $\mathrm{pFl}(1,2,\ldots,k;n)$  given by
$Z \mapsto (P_1, \ldots, P_k)$ where $P_i = \sum_{j=1}^i Z_jZ_j^T$
is $2^k$ to $1$ since $[\pm Z_1 \, \pm Z_2 \, \ldots \, \pm Z_k]$ map to the same point. We proceed as in the proof of Theorem \ref{thm:eigen-pgr-degree}. A basis for the tangent space of $\mathcal V_{k,n}$ at $Z = [e_1 \cdots e_k]$ 
consists of  $k(n-k)$ standard unit vectors $e_{ij}$ for $i=k+1, \ldots, n$ and $j=1, \ldots, k$, and the vectors $e_{ij}-e_{ji}$ for $1 \leq i < j \leq k$. The Jacobian of the above $2^k$-to-$1$ parametrization map consists of a stack of $k$ Jacobians where each individual Jacobian is the Jacobian of the parametrization map from $\mathcal V_{j,n}$ for $j=1, \ldots, k$ as we computed in the proof of Theorem \ref{thm:eigen-pgr-degree}. A careful computation shows that the images of the $k(n-k) + \binom{k}{2}$ vectors under the Jacobian of the parametrization map stay linearly independent. Hence, this image has dimension $\dim({\rm pFl}(1,2,\ldots, k;n))$. Therefore, the differential of the parametrization map is surjective. Lemma \ref{lem:critical} implies that the critical points on $\mathcal V_{k,n}$ are mapped to the critical points on ${\rm pFl}(1,2,\ldots, k;n)$. 
\end{proof}

\begin{corollary}
    The degree of the heterogeneous quadratics minimization problem over the  flag variety ${\rm pFl}(1, 2, \ldots, k; n)$ is equal to the LO degree of ${\rm pFl}(1, 2, \ldots, k; n)$. 
    Therefore, both \eqref{heterogeneous} and \eqref{heterogeneous-projection}
 have finitely many critical points. 
 \end{corollary}

\section{Two Problems from Statistics}\label{sec:statisticsproblems}
In this section, we discuss two  optimization problems from statistics which involve flags, namely \emph{canonical correlation analysis} and \emph{correspondence analysis}. Our goal is to describe and count the number of complex critical points of these optimization problems. 
The formulations are taken from \cite[Section 1.2]{YWL}.

\subsection{Canonical correlation analysis}
Canonical correlation analysis is a technique in statistics for pairing up corresponding parts of a pair of data sets \cite{MKB}. 
The problem is formulated as follows. 
Let $X, Y$ be $n \times p$ and $n \times q$ data matrices where $n$ is the common sample size.
Let $S_X, S_Y, S_{XY}$ denote the sample covariance matrices. 
The $k$th pair $(a_k,b_k) \in \RR^p \times \RR^q$ of canonical correlation loadings is 
\begin{align*}
    (a_k, b_k) = {\rm argmax}\{&a^T S_{XY}b
    \colon a^T S_X a = b^T S_Y b = 1, \\
    &a^T S_X a_j = a^T S_{XY} b_j 
    = b^T S_{YX} a_j = b^T S_Y b_j = 0, 
    \, \, j  = 1, \ldots, k-1\}.
\end{align*}
We perform a standard simplification via the Cholesky factorization $S_X = P^TP$ and $S_Y = Q^TQ$ where $P, Q$ are upper triangular. 
We substitute $A = P^{-T} S_{XY} Q^{-1}$, $u = Pa$, and $v = Qb$ to obtain the simpler problem 
\begin{align*}    
    (u_k, v_k) = {\rm argmax}\{&u^T A v
    \colon u^Tu = v^Tv = 1, \\ \nonumber
    &u^Tu_j = u^T A v_j 
    = v^T A^T u_j = v^Tv_j = 0, 
    \, \, j  = 1, \ldots, k-1\}.
\end{align*}
If we collect $u_1, \ldots, u_k$ and $v_1, \ldots, v_k$ into the $p \times k$ matrix $U$ and $q \times k$ matrix $V$, respectively, then $(U, V)$ represents a pair of flags in ${\rm Fl}(1,2, \ldots, k;p) \times {\rm Fl}(1,2, \ldots, k;q)$ \cite{YWL}.

We say the pair $(U, V) \in \CC^{p \times k} \times \CC^{q \times k}$ is a critical point of the canonical correlation problem if for all $i = 1, \ldots, k$, the pair $(u_k, v_k)$ is a critical point of the optimization problem 
\begin{align}\label{eqn:ccabrokendown}
    \textrm{maximize } &u^T A v\\\nonumber
    \textrm{subject to } &u^Tu = v^Tv = 1\\\nonumber
    &u^Tu_j = u^TAv_j = 0, \,\, j=1, \ldots, k-1\\\nonumber
    &v^TA^Tu_j = v^Tv_j = 0,  \,\, j =1, \ldots, k-1.
\end{align}
The critical points are given by the singular value decomposition of $A$ \cite{CG22}. 
\begin{theorem} \label{thm:canonical}
The critical points of the canonical correlation analysis problem are the pairs $(U, V) \in \CC^{p \times k} \times \CC^{q \times k}$
where the columns of $U$ are the left singular unit vectors of $A$, the columns of $V$ are the right singular unit vectors of $A$, and  corresponding columns have the same singular value.
There are $\binom{\min(p, q)}{k} k! 2^k $ critical points.
\end{theorem}
\begin{proof}
    The count follows from the description of the critical points: choose $k$ unit left and right singular vectors of $A$, permute the corresponding vectors simultaneously, and flip  signs of them as desired. 
    We will prove by induction that for each $i$, there exist $\lambda_i, \eta_i$ such that $Av_i = \lambda_i u_i$ and $A^T u_i = \eta_i v_i$ if and only if the pair $(u_i, v_i)$ is a critical point of \eqref{eqn:ccabrokendown}. 
    
    We proceed by induction on $i$.
    When $i = 1$, the optimization problem simplifies to 
    $\max_{u^Tu = v^Tv = 1} \,\, v^TAu.$
    By computing the transpose of the Jacobian, we find that the critical points of this problem are characterized by the leftmost column of the matrix
    \begin{align*}
        \begin{pmatrix}
            Av & u & 0\\
            A^T u & 0  & v
        \end{pmatrix}
    \end{align*}
    being in the span of the rest of the columns. 
    We therefore have that $(u,v)$ is a critical point if and only if there exist $\lambda_1, \eta_1$ such that $Av = \lambda_1 u$ and $A^Tu = \eta_1 v$. 

    Suppose now that for $j = 1, \ldots, i - 1$, we have that $Av_j = \lambda_j u_j$ and $A^T u_j = \eta_j v_j$. 
    The transpose of the augmented Jacobian matrix of \eqref{eqn:ccabrokendown} is 
    \begin{align*}
    &\left (
     \begin{array}{ccccccccccccccc}
            Av_i & u_1 & \cdots & u_i & 0 & \cdots & 0 & Av_1 & \cdots & Av_{i-1}& 0 & \cdots & 0\\
            A^Tu_i & 0 & \cdots & 0 & v_1 & \cdots & v_i & 0 & \cdots & 0 & A^Tu_1 & \cdots & A^Tu_{i-1}
        \end{array}
        \right)\\
        = &   \left (
     \begin{array}{ccccccccccccccc}
            Av_i & u_1 & \cdots & u_i & 0 & \cdots & 0 & \lambda_1 u_1 & \cdots & \lambda_{i-1}u_{i-1}& 0 & \cdots & 0\\
            A^Tu_i & 0 & \cdots & 0 & v_1 & \cdots & v_i & 0 & \cdots & 0 & \eta_1v_1 & \cdots & \eta_{i-1}v_{i-1}
        \end{array}
        \right).
     \end{align*}
    Proving that this matrix drops rank is equivalent to proving that the matrix
    \begin{align*}
         \begin{pmatrix}
            Av_i  & u_1 & \cdots & u_i & 0 & \cdots & 0\\
            A^T u_i & 0 & \cdots & 0 & v_1 & \cdots & v_i
            \end{pmatrix}
    \end{align*}
    drops rank. 
    It is clear that the matrix drops rank if $Av_i \in {\rm span}(u_i)$ and $A^Tu_i \in {\rm span}(v_i)$.
    Conversely, suppose that
    \begin{align*}
        Av_i = \alpha_1 u_1 + \cdots + \alpha_i u_i,\\
        A^T u_i = \beta_1 v_1 + \cdots + \beta_i v_i.
    \end{align*}
    Multiplying the first equation by $u_j^T$ for $j = 1, \ldots, i-1$ gives $0 = u_j^T A v_i = \alpha_j$. 
    By symmetry, $\alpha_j = \beta_j = 0$ for $j =1, \ldots, i-1$. 
    Thus $Av_i = \alpha_i u_i$ and $A^T u_i = \beta_i v_i$. 
\end{proof}

\subsection{Correspondence Analysis} 
Correspondence analysis (CA) is a statistical optimization problem over a pair of Grassmannians; see \cite{CG22, YWL}. 
CA is an analog of principal component analysis for categorical data. 
The data for CA come in the form of an $n \times p$ matrix $X$ known as a contingency table. 
We let $\mathbf{1}$ be the all-ones vector of appropriate size and set $t = \mathbf{1}^T X{\bf 1} \in \RR$. 
The row and column weights are defined as 
\begin{align*}
    r = \frac 1 t X {\bf 1} \in \RR_{}^n &&
    c = \frac 1 t X^T {\bf 1} \in \RR_{}^p\\
     D_r = \frac 1 t {\rm diag}(r) \in \RR_{}^{n \times n} &&
    D_c = \frac 1 t {\rm diag}(c) \in \RR_{}^{p \times p}.
\end{align*}
For $k = 1, \ldots, p$, we seek a pair of matrices $(U_k, V_k) \in \RR^{n \times k} \times \RR^{p \times k}$ such that 
\begin{align*}
    (U_k, V_k) = {\rm argmax}\{{\rm trace}(U^T(\frac 1 t X - rc^T)V) \colon U^TD_rU = V^TD_cV = {\rm Id}_k\}.
\end{align*}
We begin with two simplifications of this problem. 
The first is standard in correspondence analysis: since $D_r, D_c$ are diagonal, they can be factored into $U, V$, respectively, by replacing $U$ with $\sqrt{D_r}U$ and $V$ with $\sqrt{D_c}V$. 
The second is to replace the matrix $\frac 1 t X - rc^T$ from statistics with a generic matrix $A \in \RR^{n \times p}$.
The new optimization problem is
\begin{align}\label{eq:CA}
    \max_{U^TU = V^TV = {\rm Id}_k} &{\rm trace}(U^TAV).
\end{align}
The solution to this problem is given by the singular value decomposition of $A$; see \cite{CG22}. 
The problem is invariant under a simultaneous ${\rm O}(k)$-action on $U$ and $V$; hence it is an optimization problem over the product of Grassmannians ${\rm Gr}(k, n) \times {\rm Gr}(k, p)$.

\begin{theorem}\label{thm:svd}
  Let $A$ be a generic real $n \times p$ matrix with $p>n$, let $U$ be an $n \times k$ variable matrix and let $V$ be a $p \times k$ variable matrix. 
  The algebraic set of complex critical points of the optimization problem \eqref{eq:CA} 
  is equal to 
    \begin{align}\label{eq:cacp}
     \bigsqcup_{\{i_1, \ldots, i_k\} \in \binom{[n]}{k}} \{([u_{i_1} \, u_{i_2} \, \cdots \, u_{i_k}] Q,\,\, [v_{i_1} \, v_{i_2} \, \cdots \, v_{i_k}] Q) \, : \, Q \in {\rm O}(k)\} 
  \end{align}
  where $u_1, \ldots, u_n$ is an orthonormal basis of left singular vectors for $A$ and $v_1, \ldots, v_n$ is an orthogonal basis of right singular vectors for $A$ such that $u_i, v_i$ share a common singular value for $i =1, \ldots, n$. 
  This algebraic set is a disjoint union of $\binom{n}{k}$ 
  varieties isomorphic to ${\rm O}(k)$.
\end{theorem}
\begin{proof}
As in the proof of Theorem~\ref{thm:eigen}, a pair $(U, V)$ is a critical point if and only it satisfies $U^TU = V^TV = {\rm Id}_k$ and there exist symmetric matrices $M, N$ such that $AV = UM$ and $A^TU = VN$. 
We must have $M = N$, since $U^TAV = U^TUM = M$ and $V^TA^TU = V^TVN = N$ and $M, N$ are symmetric. 
Thus $(U,V)$ is a critical point if and only if $U^TU = V^TV = {\rm Id}_k$ and there exists a single symmetric matrix $M$ with $AV = UM$ and $A^TU = VM$. 
The matrix $M$ satisfies $AA^TU = UM^2$ where $AA^T$ is full rank so by Lemma~\ref{lem:diagonalization}, $M^2$ is orthogonally diagonalizable, which implies $M$ is orthogonally 
diagonalizable.
We write $M = Q^T \Lambda Q$ for the spectral decomposition of $M$. 
Then the constraints become $AVQ^T = UQ^T \Lambda$ and $A^TUQ^T = VQ^T\Lambda$ which is precisely what it means for $(UQ^T, VQ^T)$ to be in the set \eqref{eq:cacp}.
\end{proof}

The following problem  is equivalent to \eqref{eq:CA}:  
\begin{align} \label{eq:CA_SVD}
    \max_{MM^T \in {\rm pGr}(k,p)} &{\rm trace}(AM)
\end{align}
where we use the map $(U,V) \mapsto M = VU^T \in \mathbb{R}^{p \times n}$. Note that if a matrix 
$M \in \mathbb{R}^{n \times p}$ is such that 
$MM^T$ is in ${\rm Gr}(k,p)$, then the singular values of $M$ consist of $k$ ones. Therefore $M$
has a truncated singular value decomposition $M = V{\rm Id}_k U^T$ where $(U,V) \in \mathcal V_{k,n} \times \mathcal V_{k,p}$.
\begin{corollary}
The algebraic degree of \eqref{eq:CA_SVD} is $\binom{n}{k}$. The critical points are matrices $$M=[v_{i_1} v_{i_2} \cdots v_{i_k}] [u_{i_1} u_{i_2} \cdots u_{i_k}]^T$$ 
where $u_1, \ldots, u_n$ is an orthonormal basis of left singular vectors for $A$ and $v_1, \ldots, v_n$ is an orthogonal basis of right singular vectors for $A$ such that $u_i, v_i$ share a common singular value for $i =1, \ldots, n$. 
\end{corollary}
\begin{proof}
We proceed using Lemma~\ref{lem:critical}.
Because there is a transitive ${\rm O}(n) \times {\rm O}(p)$-action on the product of Stiefel manifolds $\mathcal V_{k,n} \times \mathcal V_{k,p}$, it suffices to consider the tangent space of $\mathcal V_{k,n} \times \mathcal V_{k,p}$ at the point $(U,V) = ([{\rm Id}_k \,\, 0_{k \times n-k}]^T, [{\rm Id}_k \,\, 0_{k \times p-k}]^T)$.
By computing the Jacobian, one verifies that this tangent space contains the linearly independent set 
\[\{e_{ij} : i \in \{k+1, \ldots, n\}, j \in [k] \} 
\cup \{f_{ij} : i \in \{k+1, \ldots, p\}, j \in [k] \} \cup \{e_{ij} - e_{ji}: i \neq j \in [k]\}\]
where $e_{ij}, f_{ij} \in \mathbb R^{nk + pk}$ are the indicator vectors for the variables $u_{ij}$ and $v_{ij}$, respectively.
One verifies by computation that these $nk + pk - k^2- \binom{k+1}{2}$ vectors remain linearly independent after applying the differential. 
The dimension of the image of the map $(U,V) \mapsto VU^T$ is the dimension of the domain $\mathcal V_{k,n} \times \mathcal V_{k,p}$ minus the dimension of the fibers each of which is a copy of ${\rm O}(k)$; this dimension matches the number of linearly independent vectors, so the differential is surjective.
By Lemma~\ref{lem:critical}, the critical points of \eqref{eq:CA_SVD} are precisely the images of the critical points of \eqref{eq:CA} under the map $(U,V) \mapsto VU^T$ and the statement follows from Theorem~\ref{thm:svd}.
\end{proof}

\bigskip

\noindent
\footnotesize \textbf{Authors' addresses:}

\medskip 
\noindent{Department of Mathematics, University of California, Berkeley} \hfill \texttt{hannahfriedman@berkeley.edu}
\\
\noindent{Department of Mathematics, San Francisco State University } \hfill \texttt{serkan@sfsu.edu}
\end{document}